\documentclass[a4paper,10pt]{article}

\usepackage{amsmath,amssymb,amsfonts}
\usepackage{ifthen}
\usepackage{hyperref}
\usepackage[amsmath,thmmarks,hyperref]{ntheorem}

\theoremstyle{plain}
\theoremseparator{.}
\newtheorem{theorem}{Theorem}[section]
\newtheorem{proposition}[theorem]{Proposition}
\newtheorem{lemma}[theorem]{Lemma}

\theorembodyfont{}
\theoremsymbol{\ensuremath{{}_\blacksquare}}
\newtheorem{definition}[theorem]{Definition}
\newtheorem{assumption}[theorem]{Assumption}

\theoremheaderfont{\itshape}
\newtheorem{remark}[theorem]{Remark}

\theoremsymbol{\ensuremath{\square}}
\newtheorem*{proof}{Proof}

\DeclareMathOperator*{\argmin}{arg\,min}
\newcommand{\field}[1]{\mathbb{#1}}
\newcommand{\R}{\field{R}}
\newcommand{\N}{\field{N}}

\newcommand{\abs}[2][n]{\SwitchBracketsizeLeft{#1}\LeftBracketSize\lvert#2\SwitchBracketsizeRight{#1}\RightBracketSize\rvert}
\newcommand{\norm}[2][n]{\SwitchBracketsizeLeft{#1}\LeftBracketSize\lVert#2\SwitchBracketsizeRight{#1}\RightBracketSize\rVert}
\newcommand{\set}[3][b]{\SwitchBracketsizeLeft{#1}\LeftBracketSize\{#2:#3\SwitchBracketsizeRight{#1}\RightBracketSize\}}

\newcommand{\NextScriptStyle}[1]{{\scriptstyle{#1}}}
\newcommand{\NextScriptScriptStyle}[1]{{\scriptscriptstyle{#1}}}
\newcommand{\NextTextStyle}[1]{{\textstyle{#1}}}
\newcommand{\NextDisplayStyle}[1]{{\displaystyle{#1}}}
\newcommand{\SwitchBracketsizeLeft}[1]{
  \ifthenelse{\equal{#1}{b}\OR\equal{#1}{big}}{\let\LeftBracketSize=\bigl}{
    \ifthenelse{\equal{#1}{B}\OR\equal{#1}{Big}}{\let\LeftBracketSize=\Bigl}{
      \ifthenelse{\equal{#1}{g}\OR\equal{#1}{bigg}}{\let\LeftBracketSize=\biggl}{
    \ifthenelse{\equal{#1}{G}\OR\equal{#1}{Bigg}}{\let\LeftBracketSize=\Biggl}{
      \ifthenelse{\equal{#1}{s}\OR\equal{#1}{small}}{\let\LeftBracketSize=\NextScriptStyle}{
        \ifthenelse{\equal{#1}{ss}}{\let\LeftBracketSize=\NextScriptScriptStyle}{
          \ifthenelse{\equal{#1}{t}\OR\equal{#1}{text}}{\let\LeftBracketSize=\NextTextStyle}{
        \ifthenelse{\equal{#1}{d}\OR\equal{#1}{display}}{\let\LeftBracketSize=\NextDisplayStyle}{
          \ifthenelse{\equal{#1}{a}\OR\equal{#1}{auto}}{\let\LeftBracketSize=\left}{
            \let\LeftBracketSize=\relax}}}}}}}}}}
\newcommand{\SwitchBracketsizeRight}[1]{
  \ifthenelse{\equal{#1}{b}\OR\equal{#1}{big}}{\let\RightBracketSize=\bigr}{
    \ifthenelse{\equal{#1}{B}\OR\equal{#1}{Big}}{\let\RightBracketSize=\Bigr}{
      \ifthenelse{\equal{#1}{g}\OR\equal{#1}{bigg}}{\let\RightBracketSize=\biggr}{
    \ifthenelse{\equal{#1}{G}\OR\equal{#1}{Bigg}}{\let\RightBracketSize=\Biggr}{
      \ifthenelse{\equal{#1}{s}\OR\equal{#1}{small}}{\let\RightBracketSize=\NextScriptStyle}{
        \ifthenelse{\equal{#1}{ss}}{\let\RightBracketSize=\NextScriptScriptStyle}{
          \ifthenelse{\equal{#1}{t}\OR\equal{#1}{text}}{\let\RightBracketSize=\NextTextStyle}{
        \ifthenelse{\equal{#1}{d}\OR\equal{#1}{display}}{\let\RightBracketSize=\NextDisplayStyle}{
          \ifthenelse{\equal{#1}{a}\OR\equal{#1}{auto}}{\let\RightBracketSize=\right}{
            \let\RightBracketSize=\relax}}}}}}}}}}

\title{Multi-parameter Tikhonov Regularisation in Topological Spaces}

\author{Markus Grasmair\\
\\
\normalsize
\begin{tabular}{c}
  Computational Science Center\\
  University of Vienna\\
  Nordbergstr.~15\\
  A--1090 Vienna, Austria
\end{tabular}
}

\date{September 1, 2011}

\begin{document}

\maketitle

\begin{abstract}
  We study the behaviour of Tikhonov regularisation on topological spaces
  with multiple regularisation terms.
  The main result of the paper shows that multi-parameter regularisation
  is well-posed in the sense that the results depend continuously
  on the data and converge to a true solution of the equation
  to be solved as the noise level decreases to zero.
  Moreover, we derive convergence rates in terms
  of a generalised Bregman distance using the method of variational inequalities.
  All the results in the paper, including the convergence rates,
  consider not only noise in the data, but also errors in the operator.
\end{abstract}

\section{Introduction}

Classical Tikhonov regularisation for the approximate solution
of an ill-posed operator equation $F(x)=y$ on Hilbert spaces
consists in the minimisation of the Tikhonov functional
\[
\mathcal{T}(x) := \norm{F(x)-y}^2 + \alpha \norm{x}^2
\]
for some regularisation parameter $\alpha>0$ depending on the noise level~\cite{EngHanNeu96,TikArs77}.
Here, the regularisation term $\norm{x}^2$ encodes some qualitative a--priori
knowledge about the true solution of the equation---in this case, it is assumed to have a small 
Hilbert space norm.
In many applications, however, for instance in image processing,
the a--priori knowledge has a different form than that of 
smallness of some Hilbert space norm.
Therefore it is necessary to employ other kinds of regularisation terms
(see~\cite{SchGraGroHalLen09} for an overview on regularisation methods in image processing),
and one arrives at Tikhonov functionals of the form
\[
\mathcal{T}(x) := \norm{F(x)-y}^2 + \alpha \mathcal{R}(x)
\]
with convex regularisation terms $\mathcal{R}$.
The regularising properties of Tikhonov functionals of that form
have for instance been studied in \cite{SeiVog89}.

In this paper, we study two additional generalisations
of Tikhonov regularisation.
First, we consider more general, non-quadratic, distance like measures
for the similarity term instead of the of the squared norm
of the residual. Typical examples include \emph{$f$-divergences},
which appear naturally when the noise is known to follow
a distribution that is not Gaussian.
An overview of useful similarity terms for Tikhonov regularisation
can for instance be found in~\cite{Poe08}.
The second generalisation is concerned with the regularisation term.
Instead of assuming that the a--priori information about the true
solution can be encoded in a single functional,
we study the situation where we have available different,
possibly contradicting pieces of information,
each of which can be described by the smallness of a different
regularisation term $\mathcal{R}_k$.
Examples in image processing include the 
famous Mumford--Shah model, which assumes that
images consist of different objects, characterised 
by smooth intensity variations and separated by pronounced edges~\cite{MumSha89},
but also models, which decompose images
into geometric parts and texture \cite{Mey01,VesOsh03} (see also \cite{AubKor06} for an overview).
Other examples,
where multi-parameter Tikhonov regularisation has
been proposed for the solution of inverse problems
include~\cite{BreRedRodSea03,BroAhmMacMar99,DueHof06,XuFukLiu06}.
In all these settings, regularisation is achieved by
minimisation of a Tikhonov functional of the form
\begin{equation}\label{eq:T_intro}
  \mathcal{T}(x) := \mathcal{S}(F(x),y) + \sum_k \alpha_k \mathcal{R}_k(x)\;.
\end{equation}

From the theoretical point of view,
multi-parameter Tikhonov regularisation is interesting,
as it shares most of the features of single-parameter Tikhonov regularisation,
but still exhibits some crucial differences,
for instance concerning the formulation and interpretation of convergence rates.
Still, it seems that no comprehensive study
concerning the regularising properties of multi-parameter Tikhonov regularisation
has been published.
Most of the existing theoretical papers
rather deal primarily with the problem of a suitable parameter choice
(see for instance~\cite{ItoJinTak11,LuPer11,LuPerShaTau10}),
but questions like stability and convergence have been neglected,
in particular in the interesting case when the different
regularisation parameters $\alpha_k$ decrease to zero at different rates.

In this paper, we will prove stability and convergence
of multi-parameter regularisation under fairly general conditions.
Because anyway no trace of an original Hilbert space or Banach space
structure is left in the formulation of the Tikhonov
functional $\mathcal{T}$ in~\eqref{eq:T_intro},
we will completely discard all assumption of a linear structure
and instead consider the situation, where both
the domain $X$ and the co-domain $Y$ of the operator $F$
are mere topological spaces, with the topology of $Y$
defined by the distance measure $\mathcal{S}$.
In this setting, we prove the continuous dependence
of the minimiser $x \in X$ of $\mathcal{T}$ on variations in the data $y \in Y$
and the regularisation vector $\alpha$, as long as at least one
component of $\alpha$ stays positive.
In addition, we consider the case where also
the operator $F$ to be inverted is prone to errors.
We introduce a suitable topology on the space of all
operators from $X$ to $Y$, which is related to the topology
of uniform convergence on compact sets,
and show that the minimisers of $\mathcal{T}$ also depend
continuously on the operator $F$.

Then we turn to the study of the behaviour of the minimisers
of $\mathcal{T}$ as the regularisation vector and the noise level---both
noise in the data and the operator---tend to zero.
We prove the convergence of the regularised solutions
to a solution of the exact equation provided the regularisation vector
converges to zero sufficiently slowly.
The relation between the noise level and the regularisation
that is required for deriving this kind of convergence is a
natural generalisation of the usual condition required for
the convergence of single-parameter Tikhonov regularisation.
There is, however, a notable difference to the single-parameter case.
While in the single-parameter case,
the regularised solutions converge not to any solution of the equation
$F(x) = y$, but rather to an $\mathcal{R}$-minimising one,
this behaviour cannot be guaranteed in the multi-parameter case,
if the components of $\alpha$ decrease to zero at different rates
and the regularisation terms $\mathcal{R}_k$ have different proper domains.

Finally, we derive quantitative estimates for the difference
between the regularised solution of the equation and the true solution
in dependence of the noise level, the regularisation parameter,
and the accuracy of the operator.
Because we work in general topological spaces,
we cannot employ the classical range or source conditions
for the derivation of convergence rates.
Instead, we make use of the method of variational inequalities
introduced in~\cite{HofKalPoeSch07} and its modifications and
generalisations used in~\cite{BotHof10,Gra10b}.
Apart from the extension to multi-parameter regularisation,
a major novelty of the quantitative estimate lies in the inclusion
of operator errors, which previously have never been treated
by means of variational inequalities.

\section{Preliminaries}

Assume that $X$ and $Y$ are sets and $F\colon X \to Y$ some mapping.
We consider multi-parameter Tikhonov regularisation
with a regularisation functional of the form
\begin{equation}\label{eq:T}
  \mathcal{T}(x;\alpha,y,F)
  = \mathcal{S}(F(x),y) + \sum_k \alpha_k \mathcal{R}_k(x)\,,
\end{equation}
where $\mathcal{R}_k\colon X \to [0,+\infty]$ are non-negative
regularising terms and $\mathcal{S}\colon Y \times Y \to \R_{\ge 0}$
is a distance like functional satisfying $\mathcal{S}(y,z) = 0$ if and only
if $y = z$.
Here we define $\alpha_k\mathcal{R}_k(x) := 0$ if $\alpha_k = 0$ and $\mathcal{R}_k(x) = +\infty$.

Single-parameter Tikhonov regularisation in such a general setting
with non-metric distance measure has been considered in~\cite{Fle10,FleHof10,Poe08}.
See in particular~\cite{Poe08}, where a large number of useful similarity
measures is presented.
In all these papers it was assumed that the target space $Y$ is a topological space
with a topology that is well compatible with both the function $F$
to be inverted and the similarity measure $\mathcal{S}$.
In this paper, we follow the approach from~\cite{GraHalSch09b_report},
where the well-posedness of the residual method has been treated,
and use the similarity measure $\mathcal{S}$ to define a topology on $Y$.
\medskip

First, we consider on the set $Y$ the uniformity $\mathcal{U}$ that is induced by the
family of pseudo-metrics $d^{(z)}\colon Y \times Y \to \R_{\ge 0}$, $z \in Y$, defined by
\[
d^{(z)}(y,\tilde{y}) := \abs[b]{\mathcal{S}(z,y)-\mathcal{S}(z,\tilde{y})}\;.
\]
Moreover, we denote by $\sigma$ the topology induced by the uniformity $\mathcal{U}$.
Then a sequence $\{y^{(l)}\}_{l\in\N}\subset Y$ converges to $y \in Y$
with respect to $\sigma$,
if $\mathcal{S}(z,y^{(l)}) \to \mathcal{S}(z,y)$ for every $z \in Y$.
Note in particular that the condition $\mathcal{S}(z,y) = 0$ if and only if $y = z$
implies that the topology $\sigma$ is Hausdorff,
as the metric $d^{(z)}$ separates $z$ from every other point $y \in Y \setminus \{z\}$.

\begin{remark}
  Consider the special case where the functional $\mathcal{S}$
  is of the form $\mathcal{S} = \rho \circ d$ with $d$ a metric
  on $Y$ and $\rho\colon \R_{\ge 0} \to \R_{\ge 0}$ continuous
  and strictly increasing with $\rho(0) = 0$.
  Because for every $z \in Y$ the mapping $y \mapsto \mathcal{S}(z,y) = \rho\bigl(d(z,y)\bigr)$
  is continuous with respect to the metric topology on $Y$,
  it follows that the metric topology is finer than $\sigma$.
  Conversely, every metric ball 
  \[
  B_r(z) 
  = \set{y \in Y}{d(z,y) < r}
  = \set{y \in Y}{\mathcal{S}(z,y) < \rho^{-1}(r)}
  \]
  is open with respect to $\sigma$,
  which shows that, in fact, the topology $\sigma$
  coincides with the metric topology on $Y$.
\end{remark}

\begin{remark}
  In~\cite{GraHalSch09b_report}, a topology on $Y$ has been defined by the
  single pseudo-metric
  \begin{equation}\label{eq:dsup}
    d^{\sup}(y,\tilde{y}) := \sup_{z\in Y} d^{(z)}(y,\tilde{y})
    = \sup\set[B]{\abs[b]{\mathcal{S}(z,y)-\mathcal{S}(z,\tilde{y})}}{z\in Y}\;.
  \end{equation}
  While this definition is reasonable in the case of the residual method, where
  we can assume that the functional $\mathcal{S}$ has a structure that
  is very similar to that of a distance,
  it is less so for Tikhonov regularisation, where we must rather assume
  that $\mathcal{S}$ resembles the \emph{power} of a distance.
  Indeed, already in the classical Hilbert space setting
  with $\mathcal{S}(z,y) = \norm{z-y}_Y^2$ the definition~\eqref{eq:dsup}
  is useless, as $d^{\sup}(y,\tilde{y}) = +\infty$ whenever $y \neq \tilde{y}$.
\end{remark}

In addition to the topology on $Y$, we require a suitable topology on the space $X$,
which ensures the existence of a minimiser of the
regularisation functional $\mathcal{T}(\cdot;\alpha,y,F)$
for every positive regularisation parameter
$\alpha \in \R^n_{\ge 0}\setminus\{0\}$, every $y \in Y$,
and all mappings $F$ that are compatible
with the topologies on $X$ and $Y$.
The first assumption on the topology is a standard requirement for
the subsequent application of the direct method in the calculus of variations.

\begin{assumption}\label{ass:basic}
  There exists a topology $\tau$ on $X$ such that
  each mapping $\mathcal{R}_k\colon X \to [0,+\infty]$ is sequentially coercive and 
  lower semi-continuous with respect to $\tau$.
\end{assumption}

In addition, we have to ensure that the mapping $F$
is well compatible with the topology $\tau$ on $X$.
More precisely, we require the lower semi-continuity
of the mapping $(x,y) \mapsto \mathcal{S}(F(x),y)$.
This condition alone, however, is not sufficient for obtaining
stability when one of the regularisation
parameters tends to zero (but not all of them do).
In order to treat this case as well, we have to introduce
a density condition for the domains of the regularisation
terms $\mathcal{R}_k$ with respect to a suitable type of convergence.

We denote by
\[
\mathcal{D}
:= \set{x \in X}{\mathcal{R}_k(x) < +\infty \text{ for all } 1 \le k \le n}
\]
the joint domain of the regularisation terms $\mathcal{R}_k$.
Throughout the paper, we assume that $\mathcal{D} \neq \emptyset$;
in the degenerate case $\mathcal{D} = \emptyset$,
multi-parameter regularisation with the regularisation terms $\mathcal{R}_k$
is not very useful, as, necessarily, some of the regularisation parameters $\alpha_k$
have to be zero.

\begin{definition}\label{de:ftau}
  Let $\tau$ be a topology on $X$.
  We denote by $\mathcal{F}(\tau)$ the set of all mappings
  $F\colon X \to Y$ satisfying the following conditions:
  \begin{itemize}
  \item The mapping $(x,y) \mapsto \mathcal{S}(F(x),y)$ is
    sequentially lower semi-continuous with respect to the
    product topology $(\tau \times \sigma)$ on $X \times Y$.
  \item For every $x \in X$ and every $y\in Y$ there exists a sequence 
    $\{x^{(l)}\}_{l\in\N} \subset \mathcal{D}$
    converging to $x$ with respect to $\tau$ such that $\mathcal{R}_k(x^{(l)}) \to \mathcal{R}_k(x)$
    for every $1 \le k \le n$ and $\mathcal{S}(F(x^{(l)}),y) \to \mathcal{S}(F(x),y)$.
  \end{itemize}
\end{definition}

Because we study also the stability of Tikhonov regularisation
with respect to operator errors, we have to introduce in addition
some notion of convergence of operators $F\colon X \to Y$.
To that end, we define for $K \subset X$ and $L \subset Y$ a pseudo-metric
$d_{K,L}$ on the set of all functions from $X$ to $Y$ by
\[
d_{K,L}(F,G) := \sup\set{\abs[b]{\mathcal{S}(F(x),z)-\mathcal{S}(G(x),z)}}{x \in K,\ z \in L}\;.
\]

\begin{definition}\label{de:funcconv}
  Let $F^{(l)}\colon X \to Y$, $l \in \N$, be a sequence of mappings,
  let $F\colon X \to Y$, and let $\tau$ be a topology on $X$.
  We say that the sequence $\{F^{(l)}\}_{l\in\N}$ converges to $F$
  with respect to $\tau$, if
  \[
  d_{K,L}(F^{(l)},F) \to 0
  \]
  whenever $K \subset X$ is sequentially $\tau$-compact
  and $L \subset Y$ is sequentially $\sigma$-compact.
\end{definition}

\begin{remark}
  Note that, according to Definition~\ref{de:funcconv},
  a sequence of functions $F^{(l)}\colon X \to Y$ convergence to $F\colon X \to Y$ 
  with respect to $\tau$,
  if and only if the sequence of real valued functions
  $(x,y) \mapsto \mathcal{S}(F^{(l)}(x),y)$ converges
  to the function $(x,y) \mapsto \mathcal{S}(F(x),y)$ uniformly on
  sequentially compact sets.  
\end{remark}

The following implications of
the convergence of functions introduced in~\ref{de:funcconv}
will be required several times in this paper.

\begin{lemma}\label{le:seqconv}
  Assume that the sequence of functions
  $\{F^{(l)}\}_{l\in\N}\subset \mathcal{F}(\tau)$ converges
  to $F \in \mathcal{F}(\tau)$ with respect to $\tau$.
  Assume moreover that $\{x^{(l)}\}_{l\in\N}\subset X$
  converges to $x \in X$ with respect to $\tau$
  and $\{y^{(l)}\}_{l\in\N}\subset Y$ converges to $y \in Y$
  with respect to $\sigma$.
  Then
  \[
  \lim_l \mathcal{S}(F^{(l)}(x),y^{(l)})
  = \mathcal{S}(F(x),y)
  \le \liminf_l \mathcal{S}(F^{(l)}(x^{(l)}),y^{(l)})\;.
  \]
\end{lemma}

\begin{proof}
  Consider the sequentially compact sets
  $K := \{x\}\cup\{x^{(l)}\}_{l\in\N}$ and $L := \{y\}\cup\{y^{(l)}\}_{l\in\N}$.
  Then 
  \begin{multline*}
    \abs[b]{\mathcal{S}(F^{(l)}(x),y^{(l)})-\mathcal{S}(F(x),y)}\\
    \begin{aligned}
      &\le \abs[b]{\mathcal{S}(F^{(l)}(x),y^{(l)})-\mathcal{S}(F(x),y^{(l)})}
      + \abs[b]{\mathcal{S}(F(x),y^{(l)})-\mathcal{S}(F(x),y)}\\
      &\le d_{x,L}(F^{(l)},F) + \abs[b]{\mathcal{S}(F(x),y^{(l)})-\mathcal{S}(F(x),y)}\;.
    \end{aligned}
  \end{multline*}
  Now the convergence of the sequence $\{F^{(l)}\}_{l\in\N}$ to $F$ implies that
  $d_{x,L}(F^{(l)},F) \to 0$.
  In addition, the convergence $y^{(l)}\to y$ implies that
  $\mathcal{S}(F(x),y^{(l)})\to\mathcal{S}(F(x),y)$.
  This shows that
  \[
  \lim_l \mathcal{S}(F^{(l)}(x),y^{(l)}) = \mathcal{S}(F(x),y)\;.
  \]
  Moreover we have
  \[
  \begin{aligned}
    \mathcal{S}(F^{(l)}(x^{(l)}),y^{(l)})
    &= \mathcal{S}(F(x^{(l)}),y^{(l)}) + \mathcal{S}(F^{(l)}(x^{(l)}),y^{(l)})-\mathcal{S}(F(x^{(l)}),y^{(l)})\\
    &\ge \mathcal{S}(F(x^{(l)}),y^{(l)})-d_{K,L}(F^{(l)},F)\;.
  \end{aligned}
  \]
  Thus the convergence of $\{F^{(l)}\}_{l\in\N}$ to $F$
  and the fact that $F \in \mathcal{F}(\tau)$ imply that
  \[
  \begin{aligned}
    \liminf_l \mathcal{S}(F^{(l)}(x^{(l)}),y^{(l)})
    &\ge \liminf_l \mathcal{S}(F(x^{(l)}),y^{(l)}) - \lim_l d_{K,L}(F^{(l)},F)\\
    &\ge \mathcal{S}(F(x),y)\;.
  \end{aligned}
  \]
\end{proof}

\section{Well-posedness}

In this section we study the well-posedness of multi-parameter
Tikhonov regularisation with the Tikhonov functional $\mathcal{T}$
given in~\eqref{eq:T}.
First we prove the existence of a minimiser,
and we show that the minimisers depend continuously on the
data $y\in Y$, the regularisation parameter $\alpha \in \R^n_{\ge 0}$,
and the operator $F\colon X \to Y$.
Then we prove the convergence of the minimisers
to a solution of the equation $F(x) = y$
as the noise level approaches zero,
as long as the regularisation parameters $\alpha_k$ tend to zero sufficiently slowly.

\subsection{Existence and Stability}

\begin{proposition}
  Let Assumption~\ref{ass:basic} be satisfied,
  let $F \in \mathcal{F}(\tau)$, $y \in Y$,
  and $\alpha \in \R_{\ge 0}^n\setminus\{0\}$.
  Then the Tikhonov functional $\mathcal{T}(\cdot;\alpha,y,F)$
  admits a minimum in $X$.
\end{proposition}

\begin{proof}
  This is a straightforward application of the direct
  method in the calculus of variations.
\end{proof}

The main difficulties in the proof of the stability of the regularisation
method are due to the incorporation of operator errors,
but also due to the fact
that we do not exclude the situation, where some of the regularisation
parameters $\alpha_k$ vanish (though we require that at least
one of them stays positive).

\begin{proposition}\label{pr:stab}
  Let Assumption~\ref{ass:basic} be satisfied,
  let $\{y^{(l)}\}_{l\in\N}\subset Y$ be any sequence converging to $y^\delta \in Y$
  with respect to $\sigma$,
  let $\{F^{(l)}\}_{l\in\N} \subset \mathcal{F}(\tau)$ be any sequence
  converging to $F^\delta \in \mathcal{F}(\tau)$,
  and let $\{\alpha^{(l)}\}_{l\in\N} \subset \R_{\ge 0}^n\setminus\{0\}$ converge
  to $\alpha \in \R_{\ge 0}^n\setminus \{0\}$,
  and let
  \[
  x^{(l)} \in \argmin\set{\mathcal{T}(x;\alpha^{(l)},y^{(l)},F^{(l)})}{x \in X}\;.
  \]
  Then the sequence $\{x^{(l)}\}_{l\in\N}$ has a sub-sequence
  that converges with respect to $\tau$ to some
  \[
  x_\alpha \in \argmin\set{\mathcal{T}(x;\alpha,y^\delta,F^\delta)}{x \in X}\;.
  \]
\end{proposition}

\begin{proof}
  Let $\tilde{x} \in \mathcal{D}$ be arbitrary
  and let $1 \le k_0 \le n$ be such that $\alpha_{k_0} > 0$.
  By the minimality of $x^{(l)}$ we have
  \[
  \begin{aligned}
  \alpha_{k_0}^{(l)}\mathcal{R}_{k_0}(x^{(l)})
  &\le \mathcal{T}(x^{(l)};\alpha^{(l)},y^{(l)},F^{(l)})\\
  &\le \mathcal{T}(\tilde{x};\alpha^{(l)},y^{(l)},F^{(l)})\\
  &= \mathcal{S}(F^{(l)}(\tilde{x}),y^{(l)})+\sum_k \alpha_k^{(l)}\mathcal{R}(\tilde{x})\;.\\
  \end{aligned}
  \]
  Because the sequence $\{\alpha^{(l)}\}_{l\in\N}$ converges
  to $\alpha$ and $\alpha_{k_0}>0$, we may assume without loss of generality that
  $\alpha_{k_0}^{(l)} > 0$ for all $l \in \N$,
  and therefore also $\inf_l \alpha_{k_0}^{(l)} > 0$.
  In addition, it follows that $\sup_l \alpha_k^{(l)} < +\infty$ for every $k$.
  Moreover, the convergence of the sequence $\{y^{(l)}\}_{l\in\N}$ to $y^\delta$ implies that
  $\mathcal{S}(F^{(l)}(\tilde{x}),y^{(l)}) \to \mathcal{S}(F(\tilde{x}),y^\delta) < +\infty$
  (see Lemma~\ref{le:seqconv}),
  showing that $\sup_l \mathcal{S}(F^{(l)}(\tilde{x}),y^{(l)}) < +\infty$.
  Therefore
  \[
  \sup_l\mathcal{R}_{k_0}(x^{(l)})
  \le \sup_l\frac{\mathcal{S}(F^{(l)}(\tilde{x}),y^{(l)}) + \sum_k \alpha_k^{(l)}\mathcal{R}_k(\tilde{x})}{\alpha_{k_0}^{(l)}} < +\infty\;.
  \]
  The sequential $\tau$-coercivity of $\mathcal{R}_{k_0}$ implies now the existence
  of a sub-sequence, for simplicity again denoted by $\{x^{(l)}\}_{l\in\N}$,
  converging with respect to $\tau$ to some $x_0 \in X$.

  It remains to show that $\mathcal{T}(x_0;\alpha,y^\delta,F^\delta) \le \mathcal{T}(x;\alpha,y^\delta,F^\delta)$
  for every $x \in X$.
  To that end, note first that Lemma~\ref{le:seqconv} implies that
  \[
  \mathcal{S}(F^\delta(x_0),y^\delta)
  \le\liminf_l \mathcal{S}(F^{(l)}(x^{(l)}),y^{(l)})\;.
  \]
  Because the functionals $\mathcal{R}_k$ are $\tau$-lower semi-continuous and
  $\alpha^{(l)} \to \alpha$, it follows that also
  \[
  \sum_k \alpha_k \mathcal{R}_k(x_0)
  \le \liminf_l \alpha_k^{(l)}\mathcal{R}_k(x^{(l)})\;.
  \]
  Thus
  \begin{equation}\label{eq:stab1}
  \mathcal{T}(x_0;\alpha,y^\delta,F^\delta)
  \le \liminf_l \mathcal{T}(x^{(l)};\alpha^{(l)},y^{(l)},F^{(l)})\;.
  \end{equation}

  Now assume that $\tilde{x} \in \mathcal{D}$ is arbitrary.
  Then we obtain from Lemma~\ref{le:seqconv} the equality
  \[
  \mathcal{S}(F^\delta(\tilde{x}),y^\delta) = \lim_l \mathcal{S}(F^{(l)}(\tilde{x}),y^{(l)})\;.
  \]
  Moreover, since $\mathcal{R}_k(\tilde{x}) < +\infty$ for every $k$, it follows that
  \[
  \sum_k \alpha_k^{(l)}\mathcal{R}_k(\tilde{x}) \to \sum_k \alpha_k \mathcal{R}_k(\tilde{x})\;.
  \]
  Thus~\eqref{eq:stab1} and the minimality assumption of $x^{(l)}$ imply that
  \begin{equation}\label{eq:stab2}
    \begin{aligned}
      \mathcal{T}(x_0;\alpha,y^\delta,F^\delta)
      &\le \liminf_l \mathcal{T}(x^{(l)};\alpha^{(l)},y^{(l)},F^{(l)})\\
      &\le \liminf_l \mathcal{T}(\tilde{x};\alpha^{(l)},y^{(l)},F^{(l)})\\
      &= \liminf_l \Bigl(\mathcal{S}(F^{(l)}(\tilde{x}),y^{(l)})+\sum_k\alpha_k^{(l)}\mathcal{R}_k(\tilde{x})\Bigr)\\
      &= \mathcal{S}(F^\delta(\tilde{x}),y^\delta) + \sum_k \alpha_k\mathcal{R}_k(\tilde{x})\\
      &= \mathcal{T}(\tilde{x};\alpha,y^\delta,F^\delta)\;.
    \end{aligned}
  \end{equation}

  Now let $\tilde{x} \in X$ be arbitrary.
  Then there exists a sequence $\{\tilde{x}^{(l)}\}_{l\in\N}$ converging
  to $\tilde{x}$ with respect to $\tau$ such that
  $\tilde{x}^{(l)} \in \mathcal{D}$ for every $l$, 
  $\mathcal{R}_k(\tilde{x}^{(l)})\to\mathcal{R}_k(\tilde{x})$ for every $1 \le k \le n$,
  and $\mathcal{S}(F^\delta(\tilde{x}^{(l)}),y^\delta)\to\mathcal{S}(F^\delta(\tilde{x}),y^\delta)$.
  Consequently~\eqref{eq:stab2} implies that
  \begin{equation}\label{eq:stab3}
    \mathcal{T}(x_0;\alpha,y^\delta,F^\delta)
    \le \liminf_l \mathcal{T}(\tilde{x}^{(l)};\alpha,y^\delta,F^\delta)\\
    = \mathcal{T}(\tilde{x};\alpha,y^\delta,F^\delta)\,,
  \end{equation}
  showing that $x_0$ is indeed a minimiser of $\mathcal{T}(\cdot;\alpha,y^\delta,F^\delta)$.
\end{proof}

\subsection{Convergence}

In standard Tikhonov regularisation,
the usual condition guaranteeing convergence is
the assumption $\delta^2/\alpha \to 0$.
In the setting of multi-parameter regularisation,
this assumption obviously cannot be applied directly.
One possible remedy are the assumptions
\begin{equation}\label{eq:convmax}
  \frac{\min_k \alpha_k}{\max_k \alpha_k} > c_0 > 0
  \qquad
  \text{ and }
  \qquad
  \frac{\mathcal{S}(y,y^\delta)}{\max_k \alpha_k} \to 0\;.
\end{equation}
These assumptions, however, imply that all the regularisation parameters
converge to zero at the same speed, which does not seem reasonable.
One of the main advantages of multi-parameter regularisation is precisely
its flexibility in the parameter choice, which allows for different
rates for the different components of the parameter vector.
Instead of using~\eqref{eq:convmax},
we therefore consider, in the case of an exact operator $F$, the weaker condition
\[
\frac{\mathcal{S}(y,y^\delta)}{\sum_k \alpha_k} \to 0\;.
\]
This condition allows the different parameters to decrease at
different rates.
Moreover, it makes sense, as one can interpret the denominator,
$\sum_k \alpha_k$, as the total amount of regularisation.

\begin{theorem}\label{th:conv}
  Let Assumption~\ref{ass:basic} be satisfied, let $F \in \mathcal{F}(\tau)$ and $y \in Y$,
  and assume that there exists $x_0 \in \mathcal{D}$ such that $F(x_0) = y$.
  Let moreover $\{y^{(l)}\}_{l\in\N}\subset Y$ be any sequence converging to $y$
  with respect to $\sigma$,
  and $\{F^{(l)}\}_{l\in\N} \subset \mathcal{F}(\tau)$ any sequence
  converging to $F$ with respect to $\tau$.
  Let $L := \{y\}\cup\{y^{(l)}\}_{l\in\N}$ and
  \[
  K := \set{x \in X}{\mathcal{R}_k(x) \le \mathcal{R}_k(x_0)+1 \text{ for some } 1 \le k \le n}\;.
  \]
  Assume that
  \begin{equation}\label{eq:conv1}
    \frac{\mathcal{S}(y,y^{(l)})+d_{K,L}(F^{(l)},F)}{\sum_k \alpha_k^{(l)}} \to 0
    \qquad\text{ and }\qquad
    \sum_k \alpha_k^{(l)} \to 0
  \end{equation}
  as $l \to \infty$ and let 
  \[
  x^{(l)} \in \argmin\set{\mathcal{T}(x;\alpha^{(l)},y^{(l)},F^{(l)})}{x \in X}\;.
  \]
  Consider any sub-sequence, again indexed by $l$, for which the vectors
  \[
  \bar{\alpha}^{(l)} := \frac{\alpha^{(l)}}{\sum_k \alpha_k^{(l)}} \in [0,1]^n
  \]
  converge to some
  $\bar{\alpha} \in [0,1]^n$
  (such a sub-sequence exists because of the boundedness of the sequence).
  Then every sub-sequence of $\{x^{(l)}\}_{l\in\N}$ has a sub-sequence,
  for simplicity again denoted by $\{x^{(l)}\}_{l\in\N}$,
  $\tau$-converging to some $x^\dagger \in X$
  satisfying $F(x^\dagger) = y$.
  In addition
  \begin{equation}\label{eq:xdagger}
    \sum_k \bar{\alpha}_k \mathcal{R}_k(x^\dagger)
    \le \inf\set[B]{\sum_k\bar{\alpha}_k\mathcal{R}_k(x)}{x \in \mathcal{D},\ F(x) = y}\;.
  \end{equation}
\end{theorem}

\begin{proof}
  Because the functionals $\mathcal{R}_k$ are sequentially $\tau$-coercive,
  it follows that the set $K$ is sequentially $\tau$-compact
  as the finite union of sequentially $\tau$-compact sets.
  Moreover, the definition of $x^{(l)}$ implies that
  \[
  \begin{aligned}
    \sum_k \alpha_k^{(l)} \mathcal{R}_k(x^{(l)})
    &\le \mathcal{T}(x^{(l)};\alpha^{(l)},y^{(l)},F^{(l)})\\
    &\le \mathcal{T}(x_0;\alpha^{(l)},y^{(l)},F^{(l)})\\
    &= \mathcal{S}(F^{(l)}(x_0),y^{(l)}) + \sum_k \alpha_k^{(l)} \mathcal{R}_k(x_0)\\
    &\le \mathcal{S}(F(x_0),y^{(l)}) + d_{K,L}(F^{(l)},F) + \sum_k \alpha_k^{(l)} \mathcal{R}_k(x_0)\;.
  \end{aligned}
  \]
  Dividing by $\sum_k \alpha_k^{(l)}$ and using~\eqref{eq:conv1}
  and the facts that $F(x_0) = y$ and $x_0 \in \mathcal{D}$,
  we see that
  \[
  \limsup_l \sum_k\bar{\alpha}_k^{(l)} \mathcal{R}_k(x^{(l)})\\
  \le \limsup_l \sum_k \bar{\alpha}_k^{(l)} \mathcal{R}_k(x_0)
  = \sum_k \bar{\alpha}_k \mathcal{R}_k(x_0) < +\infty\;.
  \]
  Thus also
  \begin{equation}\label{eq:conv2}
    \limsup_l \sum_k \bar{\alpha}_k \mathcal{R}_k(x^{(l)})
    \le \limsup_l \sum_k\bar{\alpha}_k^{(l)} \mathcal{R}_k(x^{(l)})
    \le \sum_k \bar{\alpha}_k \mathcal{R}_k(x_0)\;.
  \end{equation}
  In particular, for $l$ sufficiently large,
  \[
  \sum_k \bar{\alpha}_k \mathcal{R}_k(x^{(l)})
  \le \sum_k \bar{\alpha}_k \mathcal{R}_k(\tilde{x}) + 1
  = \sum_k \bar{\alpha}_k \bigl(\mathcal{R}_k(\tilde{x})+1\bigr)\,,
  \]
  which in turn proves that the sequence $\{x^{(l)}\}_{l\in\N}$ is eventually contained in 
  the sequentially $\tau$-compact set $K$.
  Thus, there exists a sub-sequence,
  for simplicity again denoted by $\{x^{(l)}\}_{l\in\N}$,
  converging with respect to $\tau$ to some $x^\dagger \in K$.

  Now note that Lemma~\ref{le:seqconv} and the facts that $x^{(l)}\to x^\dagger$,
  $y^{(l)} \to y$, $F^{(l)} \to F$, and $\sum_k \alpha_k^{(l)} \to 0$ imply that
  \[
  \begin{aligned}
    \mathcal{S}(F(x^\dagger),y)
    &\le \liminf_l \mathcal{S}(F^{(l)}(x^{(l)}),y^{(l)})\\
    &\le \liminf_l \mathcal{T}(x^{(l)};\alpha^{(l)},y^{(l)},F^{(l)})\\
    &\le \liminf_l \mathcal{T}(x_0;\alpha^{(l)},y^{(l)},F^{(l)})\\
    &\le \liminf_l \Bigl(\mathcal{S}(y,y^{(l)}) + d_{K,L}(F^{(l)},F) + \sum_k \alpha_k^{(l)}\mathcal{R}_k(x_0)\Bigr)\\
    &= 0\,,
  \end{aligned}
  \]
  showing that $\mathcal{S}(F(x^\dagger),y) = 0$ and hence $F(x^\dagger) = y$.
  Moreover, the $\tau$-lower semi-continuity of the functions $\mathcal{R}_k$ 
  and the fact that~\eqref{eq:conv2} holds for every $\tilde{x} \in \mathcal{D}\cap K$
  satisfying $F(\tilde{x}) = y$ now imply that
  \begin{multline}\label{eq:conv3}
    \sum_k \bar{\alpha}_k \mathcal{R}_k(x^\dagger)
    \le \liminf_l \sum_k \bar{\alpha}_k \mathcal{R}_k(x^{(l)})
    \le \liminf_l \sum_k\bar{\alpha}_k^{(l)} \mathcal{R}_k(x^{(l)})\\
    \le \inf\set[B]{\sum_k \bar{\alpha}_k\mathcal{R}_k(\tilde{x})}{\tilde{x} \in \mathcal{D}\cap K,\ F(\tilde{x})=y}\;.
  \end{multline}
  Now the definition of $K$ implies that also~\eqref{eq:xdagger} holds,
  which concludes the proof.
\end{proof}

Note that there is a crucial difference between the convergence result
of Theorem~\ref{th:conv} and the convergence results that
can be derived for single-parameter regularisation.
There, using conditions analogous to those of Theorem~\ref{th:conv},
one can show that the limit $x^\dagger$ is an $\mathcal{R}$-minimising
solution of the equation $F(x) = y$.
The corresponding result for multi-parameter regularisation
would be that $x^\dagger$ is a $\sum_k\bar{\alpha}_k\mathcal{R}_k$ minimising
solution of the equation $F(x) = y$, that is,
\[
x^\dagger \in \argmin\set[B]{\sum_k \bar{\alpha}_k\mathcal{R}_k(x)}{x\in X,\ F(x)=y}\;.
\]
This assertion, however, need not be true in the case where
one of the limiting regularisation parameters $\bar{\alpha}_k$ vanishes,
but the approaching weighted parameters $\bar{\alpha}_k^{(l)}$ are all positive.

\section{Convergence Rates}

For the derivation of convergence rates, or, rather, quantitative estimates
for the distance between regularised and true solution,
we apply the method of variational inequalities,
which has been introduced in~\cite{HofKalPoeSch07} (see also~\cite{SchGraGroHalLen09})
and further developed in~\cite{BotHof10,FleHof10,Gra10b}.
In a Banach space setting with convex regularisation terms,
estimates have been classically derived with respect to the Bregman distance,
which measures the distance between the regularisation term
$\mathcal{R}_k$ and its affine approximation at the true solution $x^\dagger$
(see~\cite{BurOsh04}).
Here we consider the setting introduced in~\cite{Gra10b} (see also~\cite{GraHalSch09b_report})
and assume that a variational inequality is satisfied with
respect to any distance like functional, which at the same time
serves as the distance with respect to which the convergence rates are derived.

\begin{assumption}\label{ass:varineq}
  The element $x^\dagger \in \mathcal{D}$ satisfies $F(x^\dagger) = y$,
  and for every $1 \le k \le n$ there exists a function
  $D^k(\cdot;x^\dagger) \colon X \to \R_{\ge 0}$
  satisfying $D^k(x^\dagger;x^\dagger) = 0$
  and a concave and strictly increasing function
  $\Phi_k\colon \R_{>0} \to \R_{>0}$ with $\Phi_k(0) = 0$ such that
  \begin{equation}\label{eq:varineq}
    D^k(x;x^\dagger) \le \mathcal{R}_k(x) - \mathcal{R}_k(x^\dagger) + \Phi_k\bigl(\mathcal{S}(F(x),y)\bigr)
  \end{equation}
  for every $x \in X$.
\end{assumption}

In addition to the variational inequalities~\eqref{eq:varineq},
we assume in this section that the functional
$\mathcal{S}$ satisfies a quasi-triangle inequality of the form
\begin{equation}\label{eq:triangle}
  \mathcal{S}(z_1,z_2) \le s\bigl(\mathcal{S}(z_1,z_3) + \mathcal{S}(z_3,z_2)\bigr)
\end{equation}
for some $s \ge 1$ and every $z_1$,~$z_2$,~$z_3 \in Y$.
While such a quasi-triangle inequality is satisfied
in the important case where the distance measure $\mathcal{S}$ is the
power of some metric on $Y$, it need not hold for instance in the case
where $\mathcal{S}$ is some Bregman distance.

\begin{theorem}
  Let Assumption~\ref{ass:varineq} be satisfied
  and assume that $\mathcal{S}$ satisfies the quasi-triangle inequality~\eqref{eq:triangle}.
  Define the function $\Psi\colon\R_{\ge 0}^n \to \R_{\ge 0}$,
  \[
  \Psi(\alpha) := \sup_{t > 0} \Bigl(\sum_k s^2\alpha_k \Phi_k(t)-t\Bigr)\,,
  \]
  and let
  \[
  K := \set{x \in X}{\mathcal{R}_k(x) \le \mathcal{R}_k(x^\dagger) + 1 \text{ for some } 1 \le k \le n}\;.
  \]
  Let $y^\delta \in Y$ and $F^\delta \in \mathcal{F}(\tau)$ and
  \[
  x_\alpha^\delta \in \argmin\set{\mathcal{T}(x;\alpha,y^\delta,F^\delta)}{x\in X}\;.
  \]
  Then the inequality
  \[
  \sum_k \alpha_k D^k(x_\alpha^\delta;x^\dagger)
  \le sd_{K,y}(F^\delta,F) + (s+1)\mathcal{S}(y,y^\delta) + \frac{1}{s}\mathcal{S}(y^\delta,y)
  + \frac{\Psi(\alpha)}{s^2}
  \]
  holds for $d_{K,y}(F^\delta,F)$ and $\mathcal{S}(y,y^\delta)$ sufficiently small.
  In particular
  \begin{equation}\label{eq:convrate}
    D^j(x_\alpha^\delta;x^\dagger)
    \le \frac{s^3 d_{K,y}(F^\delta,F) + (s^3+s)\mathcal{S}(y,y^\delta) + s\mathcal{S}(y^\delta,y) + \Psi(\alpha)}{s^2\alpha_j}
  \end{equation}
  for every $1 \le j \le n$
  and every $\alpha \in \R^n_{\ge 0}$ with $\alpha_j > 0$.  
\end{theorem}

\begin{proof}
  Because $x_\alpha^\delta$ is a minimiser of $\mathcal{T}(\cdot;\alpha,y^\delta,F^\delta)$,
  the inequality
  \begin{equation}\label{eq:rates1}
    \mathcal{S}(F^\delta(x_\alpha^\delta),y^\delta) + \sum_k \alpha_k \mathcal{R}_k(x_\alpha^\delta)\\
    \le \mathcal{S}(F^\delta(x^\dagger),y^\delta) + \sum_k \alpha_k \mathcal{R}_k(x^\dagger)
  \end{equation}
  holds.
  Moreover the quasi-triangle inequality~\eqref{eq:triangle} and the fact that 
  $F(x^\dagger) = y$ and hence $\mathcal{S}(F(x^\dagger),y) = 0$
  imply that
  \[
  \mathcal{S}(F^\delta(x^\dagger),y^\delta) 
  \le s\bigl(\mathcal{S}(F^\delta(x^\dagger),y)+\mathcal{S}(y,y^\delta)\bigr)
  \le s\bigl(d_{K,y}(F^\delta,F)+\mathcal{S}(y,y^\delta)\bigr)
  \]
  and
  \[
  \mathcal{S}(F(x_\alpha^\delta),y^\delta)
  \le s\bigl(\mathcal{S}(F(x_\alpha^\delta),y) + \mathcal{S}(y,y^\delta)\bigr)\;.
  \]
  Combining these inequalities with~\eqref{eq:rates1}, we obtain
  \[
  \mathcal{S}(F(x_\alpha^\delta),y^\delta) + s\sum_k\alpha_k \bigl(\mathcal{R}_k(x_\alpha^\delta)-\mathcal{R}_k(x^\dagger)\bigr)
  \le s^2 d_{K,y}(F^\delta,F) + (s^2+s)\mathcal{S}(y,y^\delta)\;.
  \]
  Using the reverse triangle inequality
  \[
  \mathcal{S}(F(x_\alpha^\delta),y^\delta)
  \ge \frac{1}{s}\mathcal{S}(F(x_\alpha^\delta),y) - \mathcal{S}(y^\delta,y)
  \]
  and the variational inequality~\eqref{eq:varineq},
  we arrive at the estimate
  \begin{multline*}
    \frac{1}{s}\mathcal{S}(F(x_\alpha^\delta),y) - s \sum_k \alpha_k\Phi_k\bigl(\mathcal{S}(F(x_\alpha^\delta),y)\bigr)
    + s\sum_k \alpha_k D^k(x_\alpha^\delta;x^\dagger)\\
    \le s^2 d_{K,y}(F^\delta,F) + (s^2+s)\mathcal{S}(y,y^\delta) + \mathcal{S}(y^\delta,y)
  \end{multline*}
  Because, by definition of $\Psi$,
  \[
  \mathcal{S}(F(x_\alpha^\delta),y)-s^2\sum_k \alpha_k \Phi_k\bigl(\mathcal{S}(F(x_\alpha^\delta),y)\bigr)
  \ge -\Psi(\alpha)\,,
  \]
  we obtain
  \[
  \sum_k \alpha_k D^k(x_\alpha^\delta;x^\dagger)
  \le sd_{K,y}(F^\delta,F) + (s+1)\mathcal{S}(y,y^\delta) + \frac{1}{s}\mathcal{S}(y^\delta,y)
  + \frac{\Psi(\alpha)}{s^2}\;.
  \]
  Moreover, \eqref{eq:convrate} holds,
  because every Bregman distance $D^j$ is non-negative
  and so are the regularisation parameters $\alpha_j$.
\end{proof}

\begin{remark}
  If one wants to obtain the best possible rate
  for one specific Bregman distance $D^j(\cdot;x^\dagger)$
  with the help of~\eqref{eq:convrate}, then one should
  set all regularisation parameters $\alpha_i$ with $i \neq j$ to zero.
  Indeed, the right hand side of~\eqref{eq:convrate}
  depends on $\alpha_i$ with $i\neq j$ only through $\Psi$,
  which is a monotonically increasing function.
\end{remark}

\section{Conclusion}

In this paper, we have shown that multi-parameter Tikhonov regularisation
is a well-posed regularisation method in the sense of Tikhonov
in a fairly general setting on topological spaces
with lower semi-continuous and coercive regularisation terms
and general distance like measures as similarity term.
In addition to proving stability of the regularisation
method with respect to noise in the data and
varying regularisation parameters,
we have also derived the continuous dependence of the regularised
solutions on the operator $F\colon X \to Y$.
To that end, we have introduced a topology on the space of
mappings from $X$ to $Y$ that is compatible with
the distance measure $\mathcal{S}$ on $Y$
serving as a similarity term.
In the classical setting of bounded linear operators
between Banach spaces and the squared norm of the residual
as a similarity measure, a sequence of bounded linear mappings
converges with respect to this topology, if and only if
it converges with respect to the norm on $L(X,Y)$.

In addition, we have shown the convergence of the
regularised solutions to a true solution, if the regularisation
vector decreases to zero slowly enough in dependence of the
noise level of the data and the operator error.
Extrapolating the results from single-parameter regularisation,
one would expect that the limit $x^\dagger$ of the regularised solutions
minimises, over the set of all solutions of the equation $F(x) = y$,
a certain convex combination $\sum_k \bar{\alpha}_k \mathcal{R}_k$ 
of the regularisation terms.
It seems, however, that this need not be the case if the
different components of the regularisation vector converge
to zero at different rates.
Then one only obtains that
\[
\sum_k\bar{\alpha}_k\mathcal{R}_k(x^\dagger)
\le \sum_k \bar{\alpha}_k\mathcal{R}_k(\tilde{x})
\]
for every $\tilde{x} \in X$ satisfying $F(\tilde{x}) = y$
and $\mathcal{R}_k(\tilde{x}) < +\infty$ for every $k$.
If one of the coefficients $\bar{\alpha}_k$ equals zero,
say $\bar{\alpha}_{k_0} = 0$
this result says nothing about the behaviour of $\sum_k\bar{\alpha}_k\mathcal{R}_k$
on solutions $\tilde{x}$ of the equation $F(x)=y$ satisfying $\mathcal{R}_{k_0}(\tilde{x}) = +\infty$.

Finally, we have derived a quantitative estimate for the
difference between the regularised solution and the true
solution $x^\dagger$ under the assumption that a variational inequality
at the solution $x^\dagger$ is satisfied.
Similarly as the stability and convergence result,
also this estimate takes into account the effect of operator errors.
Also in the case of the quantitative estimates,
the difference between multi-parameter regularisation and single-parameter
regularisation becomes relevant, as the estimates in the multi-parameter
setting cannot be directly translated into optimal convergence rates,
as it is not clear, with respect to which distance this optimality should be measured.

\end{document}